\documentclass[article]{amsart}
\usepackage{cases}
\usepackage{amsmath, amsfonts, pifont, amssymb}
\usepackage{verbatim}
\usepackage{geometry}
\newtheorem{theorem}{Theorem}[section]
\newtheorem{lemma}[theorem]{Lemma}
\newtheorem{proposition}[theorem]{Proposition}

\newcommand{\R}{\mathbb{R}}

\newcommand{\f}{\frac}
\newcommand{\beq}{\begin{equation}}
\newcommand{\eeq}{\end{equation}}
\newcommand{\beqq}{\begin{equation*}}
\newcommand{\eeqq}{\end{equation*}}

\theoremstyle{definition}
\newtheorem{definition}[theorem]{Definition}

\theoremstyle{remark}

\newtheorem{conjecture}[theorem]{Conjecture}
\numberwithin{equation}{section}

\numberwithin{equation}{section}

\begin{document}

\title{On scattering for the defocusing high dimensional inter-critical NLS  }
\author{Chuanwei Gao and Zehua Zhao}
\maketitle

\begin{abstract}
In this paper, we study the critical norm conjecture for the inter-critical nonlinear Schr{\"o}dinger equation with critical index $s_c$ satisfying $\frac{1}{2}<s_c<1$ when $d\geq 5$. Under the assumption of uniform boundedness of the critical norm, we prove the global well-posedness and scattering for the Cauchy problem. We follow the standard `Concentration compactness/Rigidity method' established in \cite{KenigMerle1,KenigMerle2}, and treat three scenarios for the critical element respectively. Moreover, double Duhamel method and interaction Morawetz estimate are applied to exclude the critical element.
\end{abstract}
\bigskip

\section{introduction}

We discuss the following nonlinear Schr\"odinger initial value problem on $\mathbb{R}_t\times \mathbb{R}_x^d$:
\begin{align}\label{01}
  \begin{cases}
  (i\partial_t+\Delta_{\mathbb{R}_x^{d}})u=\lambda |u|^pu,\\
  u(0,x)=u_0(x) \in \dot{H}^{s_c}(\mathbb{R}_x^{d})
  \end{cases}
  \end{align}
 where $p=\frac{4}{d-2s_c}$, $\lambda=1$ and $d \geq 5$. In general, we call the equation \eqref{01} defocusing when $\lambda=1$, and focusing when $\lambda=-1$. In this paper, we are dedicated to dealing with the defocusing case.\vspace{2mm}

The solutions of equation \eqref{01} are invariant under the scaling transformation:
\begin{align}\label{scale}
  u(t,x)\mapsto \lambda^{\frac{2}{p}}u(\lambda^2t,\lambda x)
  \end{align}
for $\lambda>0$. This scaling invariance naturally defines a notion of criticality. To be more specific, a direct computation shows that the homogeneous $L_x^2$-based Sobolev space that is invariant under scaling \eqref{scale} is $\dot{H}^{s_c}_x(\mathbb{R}^d)$ where the critical regularity $s_c$ is given by $s_c:=\f{d}{2}-\f{2}{p}$. We call the problem \emph{mass-critical} if $s_c=0$, \emph{energy-critical} if $s_c=1$, \emph{inter-critical} if $(0<s_c<1)$ and \emph{energy-supercritical} if $s_c>1$ respectively. In this paper, we focus on the \emph{inter-critical} case, especially the case when $\frac{1}{2} <s_c<1$. Moreover, equation \eqref{01} has the following conserved quantities:
\begin{align*}
\text{Mass: }    &\quad   \mathcal{M}(u(t))  = \int_{\mathbb{R}^d} |u(t,x)|^2dx,\\
\text{   Energy:  }     &  \quad \mathcal{E}(u(t))  = \int_{\mathbb{R}^d} \frac12 |\nabla_{\mathbb{R}^d} u(t,x,y)|^2  + \frac{1}{p+2} |u(t,x)|^{p+2}dx ,\\
\text{ Momentum: } &  \quad   \mathcal{P}(u(t)) = \Im \int_{\mathbb{R}^d} \overline{u}(x,t) \nabla_{\mathbb{R}^d} u(x,t)dx.
\end{align*}
Thus, specifically, when $s_c=0$, the $L^2$-norm of the solution stays bounded according to the conservation of mass; when $s_c=1$, the $\dot{H}^1$-norm of the solution also stays bounded according to the conservation of energy and the Sobolev inequality.\vspace{2mm}

We proceed by making the notion of solution precise.

\begin{definition}[Strong solution]
  A function $ u : I\times \R^d \rightarrow \mathbb{C} $ on a non-empty time interval $0\in I $ is a strong solution to \eqref{01} if it belongs to $C_t\dot{H}_x^{s_c}(K\times \R^d )\cap L_{t,x}^{\f{d+2}{2}p}(K\times \R^d)$ for any compact interval $K\subset I$ and obeys the Duhamel formula
  \begin{align}\label{eq:duhamel}
    u(t)=e^{it\Delta}u_0-i\int_0^te^{i(t-s)\Delta}(|u|^pu)(s)ds,
  \end{align}
  for each $t\in I.$ We call $I$ the lifespan of $u$ and we say that $u $ is a maximal-lifespan solution if it cannot be extended to any strictly larger interval and we say $u$  is a global solution  if $I=\R.$
  \end{definition}

Standard techniques show that the $L_{t,x}^{\f{d+2}{2}p}$ spacetime integrity of solution $u(t,x)$ implies scattering, that is, there exists $u_{\pm}\in \dot{H}_x^{s_c}(\mathbb{R}^d)$ such that
  \begin{align*}
   \lim_{t\rightarrow \pm \infty}\|u(t)-e^{it\Delta}u_{\pm}\|_{\dot{H}^{s_c}_x(I\times\mathbb{R}^d)}=0.
  \end{align*}

Scattering is an important asymptotic behavior of solutions and it means the solution resembles the linear solution in the long time. The above fact promotes us to define the notion of scattering size and blow up  as follows:

\begin{definition}[Scattering size and blow up]
We define the scattering size of a solution $u:I\times \mathbb{R}^d\rightarrow \mathbb{C}$ to \eqref{01} by
\begin{align*}  S_I(u):=\iint_{I\times \mathbb{R}^d}|u(t,x)|^{\f{d+2}{2}p}dxdt.
\end{align*}
We say $u$ \emph{blows up} forward in time if there exists $t_0 \in I $ such that $S_{[t_0,\sup I)}(u)=\infty$; correspondingly, we say $u$ \emph{blows up} backward in time if there exists $t_0 \in I $ such that $S_{(\inf I,t_0]}(u)=\infty$.
\end{definition}

\begin{theorem}[Local well-posedness]\label{wellposed}
Let $d\geq 5$ and $1/2 <s_c<1$, for any $ u_0 \in\dot{H}^{s_c}(\mathbb{R}^d)$~and ~$t_0 \in\mathbb{R},$~there exists a unique maximal-lifespan solution ~$u:I\times \mathbb{R}^d\rightarrow\mathbb{C}$ to \eqref{01} with ~$u(t_0)=u_0$. Furthermore,
\begin{enumerate}
\item (Local existence) $I $ is  an open neighborhood of $t_0.$
\item(Blow up ) If $\sup I$ is finite, then $u$ blows up forward in time. If $\inf I$ is finite, then ~$u$~ blows up backward in time.
    \item (Scattering and wave operators) If sup ~$I=\infty $~and ~$u$~ does not blow up forward in time, then ~$u$~ scatters forward in time. That is, there exists ~$u_{+}\in \dot{H}^{s_c}(\mathbb{R}^d)$~such that
        \begin{equation}\label{scater}
          \lim_{t\rightarrow \infty}\|u(t)-e^{it\Delta}u_{+}\|_{\dot{H}^{s_c}(\mathbb{R}^d)}=0
        \end{equation}
        Conversely, for any ~$u_+\in \dot{H}^{s_c}(\mathbb{R}^d)$~there exists a unique solution to \eqref{01} defined in a neighborhood of~$ t=\infty $~ such that \eqref{scater} holds. The analogous statements hold backward in time as well.
        \item(Small data scattering) If ~$\|u_0\|_{\dot{H}^{s_c}(\mathbb{R}^d)}$~is sufficiently small, then $u$ is global and scatters with ~$S_{\mathbb{R}}(u)\lesssim\|u\|_{\dot{H}^{s_c}(\mathbb{R}^d)}^{\f{d+2}{2}p}$~
            \end{enumerate}
        \end{theorem}
\emph{Remark.} We refer to \cite{CW} for the proof of Theorem \ref{wellposed}  (see also Section 3 in \cite{M1}).\vspace{3mm}

Before addressing our main result, we introduce the background of this problem and some important results. Generally speaking, the Cauchy problem (1.1) indicates a category of NLS problems and there is a natural corresponding conjecture based on existing results as follows:

\begin{conjecture}[Critical norm conjecture]\label{cnc}
\noindent For the following defocusing Schr{\"o}dinger initial value problem,
\begin{equation}\label{eqcnc}
\aligned
(i\partial_t+ \Delta_{\mathbb{R}^{d}}) u &= F(u) = |u|^{p} u, \\
u(0,x) &= u_{0} \in \dot{H}^{s_c}(\mathbb{R}^{d})
\endaligned
\end{equation}
where $s_c \geq 0$ and $p=\frac{4}{d-2s_c}$. If the critical norm of the solution $u(t,x)$ satisfies
\begin{equation}\label{bdd}
\sup\limits_{t\in I} ||u(t)||_{\dot{H}^{s_c}(\mathbb{R}^d)} <\infty
\end{equation}
\noindent where $I$ is the lifespan of the solution $u(t)$. Then the solution to \eqref{eqcnc} is globally well-posed and scattering.
\end{conjecture}
\emph{Remark.} Conjecture \ref{cnc} is referred to as the `\emph{Critical Norm Conjecture}'. Briefly speaking, the meaning of Conjecture \ref{cnc} is: uniform boundedness of the critical norm implies scattering. In other words, there are no `\emph{Type-II blow-up solutions}' (blow-up solutions whose critical norm stays bounded but scattering norm blows up). Moreover, for the two special cases (energy-critical when $s_c=1$ and mass-critical when $s_c=0$), we can remove the assumption of boundedness of critical norm \eqref{bdd} according to the conservation laws.\vspace{3mm}

\emph{Remark.} According to Sobolev inequality, the uniform boundedness of the critical norm \eqref{bdd} implies:
\begin{equation}\label{Linfty}
\|u(t,x)\|_{L_t^\infty L_x^{\f{dp}{2}}(I\times \R^d )} <\infty.
\end{equation}

Conjecture \ref{cnc} is a famous and important problem in the area of dispersive evolution equations. It attracts intensive interests and has been studied by many people for recent decades. Now we will briefly introduce some important results.\vspace{2mm}

For the energy-critical case, the lifespan of a solution depends on the profile of the initial data as well as its norm. Thus local existence along with conserved energy quantity can not extend the local theory to the global one. Under the radial assumption, the breakthrough was made by Bourgain's monumental work \cite{Bourgain}. Since a typical Lin-Strauss Morawetz estimate to \eqref{01} with $d=3$ holds a form as follows
\begin{align}\label{linstr}
  \int_{I}\int_{\R^3}\f{|u|^{6}}{|x|} dxdt \lesssim (\sup_{t\in I}\|u(t)\|_{\dot{H}^{\f{1}{2}}})^2.
\end{align}
The left-hand side of \eqref{linstr} appears as $\dot{H}^{\f{1}{2}}$ rather than $\dot{H}^1,$  which makes \eqref{linstr} supercritical with respect to the energy. In order to overcome this difficulty, Bourgain \cite{Bourgain} introduced a space-localized variant of \eqref{linstr} as follows
\begin{align}\label{linstr1}
  \int_{I}\int_{|x|\lesssim 1}\f{|u|^{6}}{|x|} dxdt \lesssim \operatorname{E}(u)|I|^{\f{1}{2}}，
\end{align}
together with an `induction on energy' strategy, \cite{Bourgain} settled the conjecture for the case $s_c=1, d=3,4$ under the radial assumption. Subsequently Grillakis \cite{Grillakis}, using a different method, obtained a part of the results in \cite{Bourgain}, namely global existence for smooth, radial, finite energy initial data. Later, T. Tao \cite{TaoNew} extended the results to higher dimension cases, for the nonradial setting, adapting the `induction on energy' approach of Bourgain \cite{Bourgain} with the Lin-Strauss Morawetz estimate replaced by an interaction variant introduced in \cite{Iteam2}:
  \begin{gather}\label{interactive}
  \begin{split}
    -\int_I\int\int_{\R^d\times \R^d}&|u(t,x)|^2\Delta\left (\f{1}{|.|}(x-y)\right)|u(t,y)|^2dxdydt\\
    &\lesssim\|u\|_{L_t^\infty L_x^2}^2\||\nabla|^{\f{1}{2}}u\|_{L_t^\infty L_x^2}^2.
    \end{split}
  \end{gather}

 Subsequently \cite{RM,V1} extended  the results to arbitrarily  higher  dimension.  Using the compactness-concentration argument, Kenig and Merle \cite{KenigMerle2}   obtained the  minimal blowup solution. Due to its critical property, such a solution must concentrate at frequency and physical space at the same time. To be more precise, the critical elements are compact in the associated critical Sobolev space after moduling symmetries. Since then, the method has been standardized and adapted to various settings.\vspace{2mm}

As for the mass-critical setting, motivated by the success of  settling  the energy-critical problem, Tao-Visan-Zhang \cite{TVZ} handled the mass-critical case under radial assumption for  $d\geq 3$. Then Killip-Tao-Visan \cite{KTM} extended the results to $d=2$ for both defocusing and focusing cases. For the nonradial case, we refer to B. Dodson's series works \cite{Dodson1,Dodson2, Dodson3,Dodson4}. Remarkably, B. Dodson developed a strong technique, the `Long-time Stricharz estimate' which has proved to be a useful tool.\vspace{2mm}

The first to handle the problem of the inter-critical regularity was Kenig-Merle \cite{kenigmerler2} for the case $d=3,s_c=\f{1}{2}$ by making use of their pioneered concentration-compactness argument together with Lin-Strauss Morawetz estimate. It is worth noting  that since the Lin-Strauss Morawetz inequality has a scale of $\dot{H}^{\f{1}{2}}$ which is exactly suited for this setting. Then, J. Murphy in \cite{M2} extended the analogous result to $d\geq 4$. See also \cite{GMY,M1,XF,Y} for other inter-critical results. As for energy-supercritical case, we refer to \cite{DMMZ,LZ,kv2010,MMZ,M3}. \vspace{3mm}

Now we may address our main result in this paper as follows:
\begin{theorem}[Main Theorem]\label{mainthm}
\noindent Let $u:I\times \R^d$ be the maximal lifespan solution to \eqref{01} with critical index $s_c$ satisfying $\frac{1}{2} <s_c <1$ when $d\geq 5$. Additionally, we assume that the solution $u(t)$ satisfies
\begin{equation}\label{eq:8}
\sup\limits_{t\in I} ||u(t)||_{\dot{H}^{s_c}(\mathbb{R}^d)} <\infty
\end{equation}
\noindent where $I$ is the lifespan of the solution $u(t)$. Then the solution $u$ is global and scatters in the following sense: there exists ~$u_{+}\in \dot{H}^{s_c}(\mathbb{R}^d)$~such that
        \begin{equation}\label{scattering}
          \lim_{t\rightarrow \infty}\|u(t)-e^{it\Delta}u_{+}\|_{\dot{H}^{s_c}(\mathbb{R}^d)}=0.
        \end{equation}
\end{theorem}
\emph{Remark. (Main idea of proving Theorem \ref{mainthm}):} We utilize the `Concentration compactness/Rigidity' method established in \cite{KenigMerle1,KenigMerle2} to prove Theorem \ref{mainthm}. The main idea is: if the scattering statement does not hold, we can apply the concentration compactness method to obtain a minimal blowup solution (also called `critical element') which is almost periodic. Moreover, we reduce the `critical element' into three scenarios, namely, `Finite time blow-up scenario', `Low-to-high frequency cascade scenario' and `Soliton-like scenario'. At last, we exclude the `critical element' for the three scenarios respectively.\vspace{3mm}

\textbf{Reduction to almost periodic solution:}\vspace{2mm}

Now using contradiction argument, we consider the situation where  the scattering statement \eqref{scattering} fails. At this point it suffices to draw a contradiction to prove Theorem \ref{mainthm}.\vspace{2mm}

Following standard arguments, similar to Theorem 1.12 in \cite{M1}, we can obtain a `critical element' $u(t,x):I\times \mathbb{R}^d \rightarrow \mathbb{C}$ to \eqref{01} such that $u\in L^{\infty}_t \dot{H}^{s_c}_x(I\times \mathbb{R}^d)$ which is `almost periodic' in the following sense,
\begin{definition}[Almost periodic solution]
\noindent For any $\eta>0$, there exists $C(\eta)$ such that for $t\in I$,
\begin{equation}\label{eq:7}
\int_{|x-x(t)| \geq C(\eta)/N(t)}||\nabla|^{s_c}u(t,x)|^2 dx+\int_{|\xi| \geq C(\eta)N(t)} ||\xi|^{s_c}\hat{u}(t,\xi)|^2 d\xi<\eta.
\end{equation}
\end{definition}
Moreover, the critical element $u(t,x)$ blows up both forward and backward in time. Additionally, $u(t,x)$ has minimal $L^{\infty}_t \dot{H}^{s_c}_x(I\times \mathbb{R}^d)$-norm among all blowup solutions.\vspace{3mm}

Now our task is to exclude the critical element. This takes up most of the paper and is not trivial. Our idea is to discuss the scaling function $N(t)$ and exclude the minimal blowup solutions for different scenarios respectively. The advantage of this method is that we can `split' the difficulties and use the properties of the almost periodic solutions for different scenarios to overcome them respectively. We need the following theorem for the discussion.
\begin{theorem}[Reduction Theorem]\label{reduction}
\noindent
 If there is a nonzero almost periodic solution to \eqref{01}, then there exits a nonzero almost periodic solution $u$ to \eqref{01} on an interval $I$ satisfying the following two alternatives
\begin{equation}
I=\mathbb{R},N(t)\geq 1 \textmd{  for all  } t \in \mathbb{R}
\end{equation}
or
\begin{equation}
\sup\limits(I)<\infty, \lim_{t \rightarrow sup(I)} N(t)=+\infty
\end{equation}
\end{theorem}

\emph{Remark.} The proof of Theorem \ref{reduction} follows as in Theorem 3.13 in \cite{Dodsonbook}. See also Section 4 in \cite{KV3}. (We refer to Theorem 5.24 in \cite{kv2} for the mass-critical case) \vspace{3mm}

Furthermore, according to Theorem \ref{reduction}, we make a more delicate reduction for the critical element as follows:
\begin{equation}\label{eq:11}
\textmd{Finite time blow-up scenario : }\sup\limits(I)<\infty, \lim_{t \rightarrow sup(I)} N(t)=+\infty,
\end{equation}
\begin{equation}\label{eq:15}
\textmd{Low-to-high frequency cascade scenario : }I=\mathbb{R}, \limsup_{t\rightarrow +\infty} N(t)=+\infty \textmd{ or } \limsup_{t\rightarrow -\infty} N(t)=+\infty,
\end{equation}
\begin{equation}\label{so}
\textmd{Soliton-like scenario : }I=\mathbb{R},N(t) \sim 1 \textmd{  for all  } t \in \mathbb{R}.
\end{equation}

We recall that our current goal is to exclude the critical element. It suffices to show that there is no such a critical element for the above three scenarios \eqref{eq:11}, \eqref{eq:15} and \eqref{so}. Then we can draw contradictions  (specifically, see Theorem \ref{2ft}, Theorem \ref{thmforlth} and Theorem \ref{theo1}). Thus by contradiction argument, Theorem \ref{mainthm} would be proved. We will discuss the different three scenarios and exclude the critical element respectively (see section 3, section 4, section 5 and section 6).\vspace{2mm}

\emph{The remainder of the paper is organized as follows:}  In section 2, we state some basic preliminaries and useful tools; in section 3, we exclude the Finite time blow-up scenario by using no-waste Duhamel formula and mass conservation law; in section 4, we prove the negative regularity property of the almost periodic solutions when $N(t)\geq 1$ which is a crucial step for solving this problem; in section 5, we use the negative regularity property to exclude the Low-to-high frequency cascade scenario; in section 6, we exclude the Soliton-like scenario by using interaction Morawetz estimate together with  the negative regularity property of the solutions.\vspace{3mm}

\section{Preliminaries and basic tools}
In this section, we discuss notations, preliminaries and some important tools for this problem. Most of the materials can be found in some classical textbooks, notes and papers (we refer to \cite{Iteam1,Dodsonbook,TaoNew,Taylor,kv2}).\vspace{2mm}

We write $X\lesssim Y$  or $Y\gtrsim  X$ whenever $X\leq CY$  for some constant $C>0$ and use $O(Y)$ to denote any quantity $X$ such that  $|X|\lesssim Y.$
 If $X\lesssim Y$  and $Y \lesssim X$  hold simultaneously, we abbreviate that by $X\sim Y.$  Without special clarification, the implicit constant $C$  can vary from line to line.   We denote by $X\pm$  quantity of the form $X\pm\varepsilon$ for any $\varepsilon>0.$\vspace{2mm}

For any spacetime slab $I\times \R^d,$  we use $L_t^q L_x^r(I\times \R^d)$  to denote the Banach space of functions $u:I\times \R^d \rightarrow \mathbb{C}$  whose norm is
\begin{align*}
  \|u\|_{L_t^q L_x^r(I\times \R^d)}:=\left(\int_I\|u(t)\|_{L_x^r}^qdt\right)^{\f{1}{q}}<\infty,
\end{align*}
with the appropriate modification for the case $q$ or $r$  equals to infinity. When $q=r,$  for brevity, sometimes we write it as $L_{t,x}^q.$  One more thing to be noticed is that without obscurity we will use $L_t^q L_x^r$ and $L^q L^r$  interchangeably.\vspace{2mm}

We define the Fourier transform on $\R^d$ by
\begin{align*}
  \hat{f}(\xi):=(2\pi)^{-\f{d}{2}}\int_{\R^d}e^{-ix\xi}f(x)dx,
\end{align*}
and the homogeneous Sobolev norm as
\begin{align*}
  \|f\|_{\dot{H}^s(\R^d)}:=\||\nabla|^s f\|_{L_x^2(\R^d)}
\end{align*}
where
\begin{align*}
  \widehat{|\nabla|^s f}(\xi):=|\xi|^s \hat{f}(\xi).
\end{align*}
Now we are ready to discuss Littlewood-Pelay theory which is an important and useful tool for PDE.\vspace{2mm}

Let $\phi(\xi)$  be a radial bump function supported in the ball $\{\xi\in \R^d :|\xi|\leq \f{11}{10}\}$  and equals $1$  on the ball $\{\xi\in \R^d: |\xi|\leq 1\}.$  For each number $N>0,$  we define
\begin{align*}
    \widehat{P_{\leq N}f}(\xi):=&\varphi\big(\frac{\xi}{N}\big)\hat{f}(\xi), \\
    \widehat{P_{> N}f}(\xi):=&\big(1-\varphi(\frac{\xi}{N})\big)\hat{f}(\xi), \\
    \widehat{P_{N}f}(\xi):=&\big(\varphi(\frac{\xi}{N})-\varphi(\frac{2\xi}{N})\big)\hat{f}(\xi),
\end{align*}
with similar definitions for $P_{< N}$ and $P_{\geq N}$. Moreover,  we define
\begin{align*}
    P_{M<\cdot\leq N}:=P_{\leq N}-P_{\leq M},
\end{align*}
whenever $M<N$. We record two useful lemmas regarding the Littlewood-Paley operators as follows:
\begin{lemma}[Bernstein inequalities]\label{2Bineq}
\noindent For $1\leq r\leq q \leq \infty$, $s \geq 0$, we have
\begin{equation}
|||\nabla|^{\pm s} P_N f||_{L^r(\mathbb{R}^d)} \sim N^{\pm s}||P_N f||_{L^r(\mathbb{R}^d)},
\end{equation}
\begin{equation}
|||\nabla|^{ s} P_{\leq N} f||_{L^r(\mathbb{R}^d)} \lesssim N^s ||P_{\leq N} f||_{L^r(\mathbb{R}^d)},
\end{equation}
\begin{equation}
|| P_{\geq N} f||_{L^r(\mathbb{R}^d)} \lesssim N^{-s} |||\nabla|^{s} P_{\geq N} f||_{L^r(\mathbb{R}^d)},
\end{equation}
\begin{equation}
|| P_{\leq N} f||_{L^q(\mathbb{R}^d)} \lesssim N^{\frac{d}{r}-\frac{d}{q}} || P_{\leq N} f||_{L^r(\mathbb{R}^d)}.
\end{equation}
\end{lemma}

We will also need the following chain rule for fractional order derivatives. One can turn to \cite{CW2} for more details.
\begin{lemma}[Fractional chain rule]\label{chain}
Suppose ~$G\in C^1(\mathbb{C})$~and ~$s\in (0,1].$~Let ~$1<r<r_2<\infty $~and ~$1<r_1\leq \infty $~be such that~$\frac{1}{r}=\frac{1}{r_1}+\frac{1}{r_2},$~then
\begin{align}\label{fpr}
  \||\nabla|^sG(u)\|_{L^r_x}\lesssim\|G'(u)\|_{L_x^{r_1}}\||\nabla|^su\|_{L_x^{r_2}}.
\end{align}
\end{lemma}
We now record the dispersive estimate for the Schr\"odinger operator and some standardized Strichartz estimate.
\begin{lemma}[Dispersive estimate]\label{22disp}
\begin{equation}\label{disp}
||e^{it\Delta} f ||_{L^{\infty}_x(\mathbb{R}^d)} \lesssim |t|^{-\frac{d}{2}}||f||_{L^1(\mathbb{R}^d)}.
\end{equation}
\end{lemma}
\noindent \emph{Remark.} When the physical dimension of the function is higher, the decay is faster. Moreover, if we interpolate \eqref{disp} with $||e^{it\Delta}f||_{L^2(\mathbb{R}^d)}=||f||_{L^2(\mathbb{R}^d)}$ (Plancherel formula), we can obtain
\begin{equation}
||e^{it\Delta} f ||_{L^{r}_x(\mathbb{R}^d)} \lesssim |t|^{-(\frac{d}{2}-\frac{d}{r})}||f||_{L^{r^{'}}(\mathbb{R}^d)}.
\end{equation}
\noindent where $2\leq r \leq \infty$, $\frac{1}{r}+\frac{1}{r^{'}}=1$ and $t \neq 0$.

\begin{definition}[Admissible pair]\label{admi} Let $d\geq 5,$ we call a pair of exponent $(q,r)$ admissible if
\begin{align}
  \frac{2}{q}=d(\f{1}{2}-\f{1}{r}) \quad \textrm{with}\quad 2\leq q\leq \infty.
\end{align}
For a time interval $I$, we define
\begin{align}
  \|u\|_{\rm {S}(I)}:=\sup\{\|u\|_{L^q_tL^r_x(I\times\mathbb{R}^d)}:(q,r)  \text{   admissible}\}
\end{align}
We also define the dual of $ S(I)$ by $N(I)$, we note that
\begin{align}
  \|u\|_{\rm N(I)}\lesssim \|u\|_{L_t^{q'}L_x^{r'}(I\times \R^d)}\quad \text{for any admissible pair}~(q,r).
\end{align}
\end{definition}
\begin{proposition}[Strichartz estimate]
  Let ~$u:I\times \mathbb{R}^d\rightarrow \mathbb{C}$~be a solution to
  \begin{align}
    (i\partial_t+\Delta)u=F
  \end{align}
  and let~$ s\geq 0$, then
  \begin{align}
    \||\nabla|^su\|_{S(I)}\lesssim\|u(t_0)\|_{\dot{H}^s_x}+\||\nabla|^sF\|_{N(I)},
  \end{align}
  for any ~$t_0 \in I$.
\end{proposition}

The next proposition says in contrast with the classical Duhamel formula \eqref{eq:duhamel} that there is no scattered wave at the endpoint of the lifespan $I$ for almost periodic solutions. We refer to \cite{TVZ} for more information.
\begin{proposition}[No-waste Duhamel formula]\label{le1}
Let ~$u:I\times \mathbb{R}^d\rightarrow \mathbb{C}$~be a maximal-lifespan almost periodic solution to \eqref{01}, then for all~$ t \in I$,
\begin{align}
  u(t)&=\lim_{T\nearrow \sup I}i\int_{t}^Te^{i(t-s)\Delta}(|u|^pu)(s)ds\\
  &=-\lim_{T\searrow \inf I}i\int_{T}^te^{i(t-s)\Delta}(|u|^pu)(s)ds
\end{align}
as a weak limit in ~$\dot{H}^{s_c}_x(\mathbb{R}^d)$~.
\end{proposition}

Based on the above basic results and tools, we are ready to exclude the critical element for the three scenarios in the following three sections respectively.\vspace{3mm}

\section{finite time blow-up scenario }
In this section, we deal with the Finite time blow-up scenario. We will use no-waste Duhamel formula and interpolation to show
 \beqq
 \lim \limits_{t\rightarrow T_{\text{max}}}\|u(t)\|_{L_x^2}\rightarrow 0
 \eeqq in the context of \eqref{eq:11}, where $T_{max}=\sup(I)$ and $I$ is the lifespan of $u$. According to the conservation of mass, this implies $u\equiv 0$, which contradicts the fact that $u$ is a blowup solution to \eqref{01}.
\begin{theorem}[Finite time blow-up scenario]\label{2ft}
\noindent If $u: I \times \R^d $ is an almost periodic solution satisfying \eqref{eq:11}, then $u(x,t) \equiv 0$.
\end{theorem}
\begin{proof}
We argue by contradiction, assuming $ u$ is the Finite time blow-up solution to \eqref{01} with $T_{\text{max}}<\infty$.\vspace{2mm}

By Strichartz estimate, \eqref{eq:8}, Proposition \ref{le1}, Sobolev inequality, H\"older's inequality and \eqref{Linfty}, we obtain
 \begin{align*}
   \||\nabla|^{s_c-1}u(t)\|_{ L_x^2}&\lesssim \lim_{T\rightarrow T_{max}} \|\int_t^T |\nabla|^{s_c-1}e^{i(t-s)\Delta}F(u)(s)ds\|_{L_t^\infty L_x^2(I\times \R^d )}\\
   &\lesssim  \||\nabla|^{s_c-1}F(u)\|_{L_t^{2}L_x^{\frac{2d}{d+2}}((t,T_{\text{max}})\times \R^d )}\\
   &\lesssim  \|F(u)\|_{L_t^{2}L_x^{\frac{pd}{2p+2}}((t,T_{\text{max}})\times \R^d )}\\
   &\lesssim (T_{\text{max}}-t)^{\frac{1}{2}}\|u\|_{L_t^\infty L_x^{\f{dp}{2}}((t,T_{\text{max}})\times \R^d )}^{p+1}\\
   &\lesssim (T_{\text{max}}-t)^{\frac{1}{2}}
 \end{align*}
 By interpolating with
 \beqq
\||\nabla|^{s_c}u\|_{L^\infty_t L_x^{2}(I\times \R^d )}<\infty,
 \eeqq
  we see
\begin{equation}
  ||u(t)||_{ L^2_{x}} \lesssim (T_{max}-t)^{\frac{s_c}{2}}.
\end{equation}
Furthermore, when $t\rightarrow T_{max}, \|u(t)\|_{L_x^2} \rightarrow 0.$ Therefore, according to the conservation of mass,  $u\equiv 0$ which is inconsistent with the fact that $u$ is a blowup solution.

\end{proof}
\section{Negative regularity property }
For this problem, we take advantage of the high dimension setting to use double Duhamel trick to obtain negative regularity of the solution which is one of the most essential steps in this paper (see \cite{Iteam2,Dodsonbook,KV3,kv2,V1,V2}). Furthermore, we use the negative regularity and the almost periodicity to exclude the critical element for the Low-to-high frequency cascade scenario and the Soliton-like scenario. We use the following lemmas to obtain the negative regularity.

\begin{lemma}[Double Duhamel Lemma]\label{dou}
\noindent If $I$ is the maximal interval of existence and $u$ is an almost periodic solution, then for any $t_1\in I$,
\begin{equation}
\langle u(t_1),u(t_1) \rangle_{\dot{H}^{s_c}}=-\lim\limits_{t_0 \rightarrow inf(I)} \lim\limits_{t_2 \rightarrow sup(I)} \langle \int_{t_0}^{t_1}e^{i(t_1-\tau)\Delta}F(u(\tau))d\tau,\int_{t_1}^{t_2}e^{i(t_1-t)\Delta}F(u(t))dt \rangle_{\dot{H}^{s_c}}.
\end{equation}
\end{lemma}

Using Lemma \ref{dou} and no waste Duhamel formula, we can obtain `negative regularity' of the solution $u$. First we will prove an important improvement over the Sobolev embedding theorem in the sense that there exits some $p< \frac{2d}{d-2s_c}$ such that $u \in L_x^p$. Second we will use the result obtained in the first step to show there exists $s_0$ such that $u\in \dot{H}^{s_c-s_0}$. Iterating the second steps several times, we will ultimately obtain the desired result. Now we turn to the details of the procedure, first we will establish the following key lemma  which says that $u$ has additional decay.
\begin{lemma}\label{base}
\noindent Suppose $u$ is an almost periodic solution satisfying (1.7), when $d\geq 5$, for any $q$ satisfying
\begin{equation}\label{eq:5}
\frac{2(d+4-2s_c)}{d+4-4s_c}<q\leq \frac{2d}{d-2s_c}
\end{equation}
we have
\begin{equation}\label{eq:4}
||u(t)||_{L^{\infty}_{t}L^q_{x}(\mathbb{R}\times \mathbb{R}^d)} \lesssim_q 1.
\end{equation}

\end{lemma}

\begin{proof}
From \eqref{eq:8}  and Sobolev embedding, we know \eqref{eq:4} holds when $q=\frac{2d}{d-2s_c}$.   It suffices to consider the left endpoint case in \eqref{eq:5}. Note \eqref{eq:7} and the fact that we have frozen  $N(t)\geq 1$, thus for any $\eta>0$, there exists $j_0(\eta) \in \mathbb{Z}$  such that
\begin{equation}
||P_{j_0}u(t)||_{L^{\infty}_t \dot{H}^{s_c}(\mathbb{R}\times \mathbb{R}^d)} \leq \eta.
\end{equation}

\noindent Now we take $r$ satisfying $\frac{1}{r}=\frac{1}{2}-\frac{1}{m}$ where $m=\frac{d}{2-s_c}-2$. \vspace{2mm}

\noindent It suffices to prove the following estimate:
\begin{equation}\label{eq:9}
\sup_{j\leq j_0}2^{\epsilon j} 2^{-2j(\frac{d}{m}-1)} ||P_j u||_{L_t^{\infty}L^r_x} \lesssim_{\epsilon} 1.
\end{equation}
\noindent Indeed \eqref{eq:4} follows by interpolating with
 \begin{equation}
2^{js_c}||P_ju||_{L^2(\mathbb{R}^d)}<\infty.
\end{equation}

\noindent Now our goal is to prove inequality \eqref{eq:9}. First we will show:
\begin{equation}\label{eq:13}
||P_j u(t)||_{L_t^{\infty}L_x^r} \lesssim 2^{2j(\frac{d}{m}-1)}||P_j F(u(t))||_{L_t^{\infty}L_x^{r^{'}}}.
\end{equation}
\noindent For any given $t\in \R$, by no-waste Duhamel formula,  it suffices to show
\begin{align}
\|P_j u(t)\|_{L_x^r}&\leq \left\|\int_t^{t+2^{-2j}}e^{i(t-s)\Delta}P_jF(u(s))ds\right\|_{L_x^r}+\left\|\int_{t+2^{-2j}}^\infty e^{i(t-s)\Delta}P_jF(u(s))ds\right\|_{L_x^r}\\
&\lesssim 2^{2j(\frac{d}{m}-1)}||P_j F(u(t))||_{L_t^{\infty}L_x^{r^{'}}}.
\end{align}
By Bernstein's inequality \eqref{2Bineq}, we estimate
\begin{align}\label{eq:10}
\begin{split}
\left\|\int_t^{t+2^{-2j}}e^{i(t-s)\Delta}P_jF(u(s))ds\right\|_{L_x^r}&\lesssim 2 ^{jd(\frac{1}{2}-\frac{1}{r})}\left\|\int_t^{t+2^{-2j}}e^{i(t-s)\Delta}P_jF(u(s))ds\right\|_{L_x^2}\\
&\lesssim  2^{2j(\frac{d}{m}-1)}||P_j F(u(t))||_{L_t^{\infty}L_x^{r^{'}}}.
\end{split}
\end{align}
\noindent On the other hand, using dispersive estimate \eqref{disp},
\begin{align}\label{eq:12}
\begin{split}
\left\|\int_{t+2^{-2j}}^\infty e^{i(t-s)\Delta}P_jF(u(s))ds\right\|_{L_x^r}&\lesssim \int_{t+2^{-2j}}^\infty \frac{1}{|t-s|^{\frac{d}{m}}}\left\|P_jF(u(s))\right\|_{L_x^r}ds\\
&\lesssim  2^{2j(\frac{d}{m}-1)}||P_j F(u(t))||_{L_t^{\infty}L_x^{r^{'}}}.
\end{split}
\end{align}
\noindent Thus, \eqref{eq:13} follows by combing \eqref{eq:10} and \eqref{eq:12}.\vspace{2mm}

\noindent It remains to prove \eqref{eq:9}. To this end, we decompose
\begin{equation}
F(u)=F(P_{\leq j}u)+O(|P_{>j}u|\cdot |P_{\leq j}u|^{\frac{4}{d-2s_c}})+O(|P_{>j}u|^{\frac{d+4-2s_c}{d-2s_c}}).
\end{equation}
 Now we consider the case $d=5$ and the case $d\geq 6$ respectively.\vspace{2mm}

\noindent When $d=5$,
\begin{equation}\label{eq:20}
\sup_{j\leq j_0}2^{\epsilon j}  ||P_j F(P_{\leq j}u)||_{L_t^{\infty}L^{r^{'}}_x} \lesssim_{\epsilon} \eta^{\frac{4(1-\theta)}{d-2s_c}}(\sup_{j\leq j_0}2^{\epsilon j} 2^{-2(\frac{d}{m}-1)} ||P_j u||_{L_t^{\infty}L^r_x})^{\frac{4\theta }{d-2s_c}},
\end{equation}
where $\theta=\frac{m}{2}-\frac{d-2s_c}{4}$ (and $\frac{4\theta }{d-2s_c}<1$). \vspace{2mm}

\noindent Meanwhile, we can obtain
\begin{equation}
\sup_{j\leq j_0}2^{\epsilon j}||P_{>j}u||_{L^2}||P_{\leq j}u||^{\frac{4}{d-2s_c}}_{L^{\frac{4m}{d-2s_c}}} \lesssim \eta^{\frac{4(1-\theta)}{d-2s_c}}(\sup_{j\leq j_0}2^{\epsilon j} 2^{-2(\frac{d}{m}-1)} ||P_j u||_{L_t^{\infty}L^r_x})^{\frac{4\theta }{d-2s_c}},
\end{equation}
\noindent And
\begin{equation}\label{eq:21}
\sup_{j\leq j_0} 2^{\epsilon j}||F(P_{>j} u)||_{L^{r^{'}}}\lesssim_{\epsilon} \eta^{s}(\sup_{j\leq j_0}2^{\epsilon j} 2^{-2(\frac{d}{m}-1)} ||P_j u||_{L_t^{\infty}L^r_x})^{\tau}+C_{\eta}
\end{equation}
\noindent where $s>0$ and $\tau=m(\frac{1}{2}\frac{d+4-2s_c}{d-2s_c}-\frac{1}{r^{'}})<1$. According to \eqref{eq:20}-\eqref{eq:21},
\begin{equation}
\sup_{j\leq j_0}2^{\epsilon j} 2^{-2(\frac{d}{m}-1)} ||P_j u||_{L_t^{\infty}L^r_x} \lesssim C_{\eta}+C_{\epsilon}\eta^{\frac{4(1-\theta)}{d-2s_c}}(\sup_{j\leq j_0}2^{\epsilon j} 2^{-2(\frac{d}{m}-1)} ||P_j u||_{L_t^{\infty}L^r_x})^{\frac{4\theta }{d-2s_c}}.
\end{equation}
\noindent Choosing $\eta(\epsilon)>0$ sufficiently small, \eqref{eq:9} would be proved for the case $d=5$.\vspace{3mm}

\noindent When $d\geq 6$, the calculations are similar. By Bernstein inequality,
\begin{equation}\label{eq:18}
\aligned
2^{\epsilon j}||P_j F(P_{\leq j}u)||_{L^{r^{'}}}&\lesssim 2^{\epsilon j}2^{-j(1-\alpha)s_c \frac{d+4-2s_c}{d-2s_c}}|| |\nabla|^{s_c}P_{\leq j}u||^{(1-\alpha)\frac{d+4-2s_c}{d-2s_c}}_{L^2_x} ||P_{\leq j}u||^{\alpha \frac{d+4-2s_c}{d-2s_c}}_{L^r_x} \\
&\lesssim_{\epsilon}  C_{\epsilon}\eta^{(1-\alpha) \frac{d+4-2s_c}{d-2s_c}}(\sup_{k\leq j_0}2^{\epsilon k}2^{-2(\frac{d}{m}-1)k}||P_k u||_{L^r_x})^{\alpha \frac{d+4-2s_c}{d-2s_c}}.
\endaligned
\end{equation}

\noindent where $\alpha=\frac{\frac{s_c}{2-s_c}d-4+2s_c}{d+4-2s_c}$ and $\alpha \frac{d+4-2s_c}{d-2s_c}<1$.\vspace{2mm}

\noindent By interpolation and Bernstein inequality, we obtain
\begin{equation}
\aligned
&|| |P_{>j}u| |P_{<j}u|^{\frac{4}{d-2s_c}} ||_{L_x^{r^{'}}} \\
&\lesssim  ||P_{\leq j}u||^{\frac{4}{d-2s_c}}_{L^r_x}(\sum_{j \leq  k \leq j_0}||P_k u||^{\beta}_{L_x^r} ||P_k u||^{1-\beta}_{L_x^2}+\sum_{k \geq j_0}||P_k u||^{\beta}_{L_x^r} ||P_k u||^{1-\beta}_{L_x^2})  \\
&\lesssim \eta^{\beta}\left((\sup_{k\leq j_0}2^{\epsilon k}2^{-2(\frac{d}{m}-1)k}||P_k u||_{L^r_x})^{\frac{4}{d-2s_c}+\beta} + (\sup_{k\leq j_0}2^{\epsilon k}2^{-2(\frac{d}{m}-1)k}||P_k u||_{L^r_x})^{\frac{4}{d-2s_c}} \right)          ,
\endaligned
\end{equation}
\noindent where $\beta=-1+\frac{2(d-8+4s_c)}{(2-s_c)(d-2s_c)}$ and $\beta+\frac{4}{d-2s_c}<1$.\vspace{2mm}

\noindent At last,
\begin{equation}\label{eq:19}
\aligned
&(\sum_{j \leq  k \leq j_0}||P_k u||^{1-\alpha}_{L^2_x}||P_k u||^{\alpha}_{L^r_x}+\sum_{k \geq j_0}||P_k u||^{1-\alpha}_{L^2_x}||P_k u||^{\alpha}_{L^r_x})^{\frac{d+4-2s_c}{d-2s_c}} \\
&\lesssim C_{\epsilon}\eta^{(1-\alpha)s_c \frac{d+4-2s_c}{d-2s_c}}(\sup_{k\leq j_0}2^{\epsilon k}2^{-2(\frac{d}{m}-1)k}||P_k u||_{L^r_x})^{\alpha \frac{d+4-2s_c}{d-2s_c}}+C_{\eta}.
\endaligned
\end{equation}
\noindent Putting \eqref{eq:18}-\eqref{eq:19} together, we conclude that
\begin{equation}
\sup_{j\leq j_0}2^{\epsilon j} 2^{-2(\frac{d}{m}-1)} ||P_j u||_{L_t^{\infty}L^r_x} \lesssim_{\epsilon} 1,
\end{equation}
\noindent Which proves \eqref{eq:9}  for the case $d\geq 6$. This completes the proof of Lemma \ref{base}.
\end{proof}

Using double-Duhamel formula, we may upgrade \eqref{eq:4} by showing that $u$ lies in additional Sobolev space apart from critical one. From which  one can iterate the above procedure and ultimately obtain negative regularity result.
\begin{lemma}[Inductive argument]\label{ind}
\noindent Let $d\geq 5$ and $u$ be as in Lemma \ref{base}. Assume that $|\nabla|^{s}F(u)\in L_t^{\infty}L_x^{q}$ for $q=\frac{2(d+2-s_c)(d-2s_c)}{(d+2-s_c)(d-2s_c)+4(d+2-3s_c)}$ and some $0 \leq s\leq 1$. Then there exists $s_0=\frac{d}{q}-\frac{d+4}{2}>0$ such that $u\in L^{\infty}_t\dot{H}^{(s-s_0)+}_x$.
\end{lemma}

\begin{proof}
It suffices to show that for $s_0>0$ and all $j\leq 0$,
\begin{equation}
|| |\nabla|^s P_j u ||_{L^{\infty}_{t} L^2_x} \lesssim 2^{j s_0}.
\end{equation}
\noindent By time translation symmetry, it suffices show that
\begin{equation}
|| |\nabla|^s P_j u(0) ||_{L^2_x} \lesssim 2^{j s_0}.
\end{equation}
\noindent The proof uses the double Duhamel argument, by Lemma \ref{dou}
\begin{equation}
|||\nabla|^s P_j u(0)||^2_{L^2_x} \leq \int_0^{\infty} \int_{-\infty}^{0}|\langle P_j |\nabla|^s F(u(t)),e^{i(t-\tau)\Delta}P_j |\nabla|^s F(u(\tau)) \rangle|  dt d\tau
\end{equation}
\noindent Using dispersive estimate, we obtain
\begin{equation}
|\langle P_j |\nabla|^s F(u(t)),e^{i(t-\tau)\Delta}P_j |\nabla|^s F(u(\tau)) \rangle_{L^2}| \lesssim |t-\tau|^{d(\frac{1}{2}-\frac{1}{q})}|||\nabla|^s F(u)||^2_{L^{\infty}_t L^q_x} .
\end{equation}
\noindent Also by the Sobolev embedding theorem,
\begin{equation}
|\langle P_j |\nabla|^s F(u(t)),e^{i(t-\tau)\Delta}P_j |\nabla|^s F(u(\tau)) \rangle_{L^2}| \lesssim 2^{2j(\frac{d}{q}-\frac{d}{2})}|||\nabla|^s F(u)||^2_{L^{\infty}_t L^q_x} .
\end{equation}
\noindent Thus, we have
\begin{align*}
|||\nabla|^s P_j u(0)||^2_{L^2_x} &\lesssim |||\nabla|^s F(u)||^2_{L^{\infty}_t L^q_x} \int_0^{\infty} \int_{-\infty}^{0} \textmd{min}(|t-\tau|^{-1},2^{2j})^{\frac{d}{q}-\frac{d}{2}} \\ &\lesssim 2^{2j s_0}|||\nabla|^s F(u)||^2_{L^{\infty}_t L^q_x}.
\end{align*}

Now the proof of \ref{ind} is complete.
\end{proof}

We can use Lemma \ref{base} as a base case and apply  Lemma \ref{ind} inductively to obtain the negative regularity. To this end, first we need to verify that
\beqq
|\nabla|^{s}F(u)\in L_t^{\infty}L_x^{q} \quad \text{for} \quad q=\frac{2(d+2-s_c)(d-2s_c)}{(d+2-s_c)(d-2s_c)+4(d+2-3s_c)}
\eeqq
Indeed, by fractional product rule \eqref{fpr},
\beq
\|\nabla|^{s}F(u)\|_{L_x^q}\lesssim \||\nabla|^s u\|_{L_x^2}\|u\|_{L_x^{\frac{2(d+2-s_c)}{d+2-3s_c}}}^p.
\eeq
Note that
\beqq
\frac{2(d+4-2s_c)}{d+4-4s_c}<\frac{2(d+2-s_c)}{d+2-3s_c}\leq \frac{2d}{d-2s_c}，
\eeqq
then from \eqref{eq:4} we get
\beq
\|\nabla|^{s}F(u)\|_{L_x^q}<\infty.
\eeq
Ultimately one may apply Lemma \ref{ind} several times to obtain:
\begin{theorem}[Negative regularity]\label{theo2}
\noindent Suppose $d\geq 5$, and $u$ is an almost periodic solution satisfying (1.7) then $u \in L^{\infty}_t\dot{H}^{-\epsilon(d)}_x(\mathbb{R}\times \mathbb{R}^d)$ for some $\epsilon=\epsilon(d)>0$. In particular, this implies $u \in L^{\infty}_tL^{2}_x(\mathbb{R}\times \mathbb{R}^d)$.
\end{theorem}

\section{low-to-high frequency cascade scenario }
 In this section, we deal with the Low-to-high frequency cascade scenario. Using the negative regularity property together with conservation of mass, we can exclude the Low-to-high frequency cascade scenario as follows:
\begin{theorem}[Low-to-high frequency cascade scenario]\label{thmforlth}
If $u$ is an almost periodic solution satisfying \eqref{eq:15}, then $u(t,x)\equiv 0$.
\end{theorem}
 \begin{proof}
From \eqref{eq:7},  for any $\eta>0$, there exists $C(\eta)$  such that
\begin{equation}\label{eq:16}
||P_{\leq C(\eta)N(t)}u(x,t)||_{\dot{H}^{s_c}}<\eta.
\end{equation}
Meanwhile, by Theorem \eqref{theo2}
\begin{equation}\label{eq:17}
\sup_{t\in \mathbb{R}} ||u(x,t)||_{\dot{H}_x^{-\epsilon(d)}(\mathbb{R}^d)} <+\infty,
\end{equation}
We interpolate \eqref{eq:16} and \eqref{eq:17} to obtain
\begin{equation}
||P_{\leq C(\eta)N(t)}u(t)||_{L^2}<\eta^s,
\end{equation}
where $0<s<1$. Furthermore, by Bernstein inequality and \eqref{eq:8},
\begin{equation}
||P_{\geq C(\eta)N(t)}u(t)||_{L^2(\mathbb{R}^d)}\lesssim \frac{1}{(C(\eta)N(t))^{s_c}}.
\end{equation}
\noindent Thus,
\begin{equation}
||u(t)||_{L^2(\mathbb{R}^d)}\leq ||P_{\leq C(\eta)N(t)}u(t)||_{L^2}+||P_{\geq C(\eta)N(t)}u(t)||_{L^2}\lesssim \eta^s+\frac{1}{(C(\eta)N(t))^{s_c}}.
\end{equation}
Since $\eta$ can be chosen arbitrarily small and let $t \rightarrow +\infty$, by the conservation of mass law,  we obtain $||u(t,x)||_{L^2} \equiv 0$, which implies $u\equiv 0$.
\end{proof}
\section{Soliton-like scenario}
In this section, we take care of the last scenario, i.e. Soliton-like scenario. Combining the negative regularity property obtained in section $4$ with interaction Morawetz estimate, the Soliton-like solution in the context of \eqref{so} can be excluded. We recall the following result in \cite{TVZ2}:

\begin{theorem}[Interaction Morawetz estimate \cite{TVZ2}]
\noindent If $u$ solves \eqref{01}
\noindent on $I\times \mathbb{R}^d$ for some $p>0$ and $d\geq 5$, then
\begin{equation}\label{inter}
|||\nabla|^{\frac{3-d}{4}}u||_{L^4_{t,x}(I\times \R^d )} \lesssim ||u||^{\frac{1}{2}}_{L_t^{\infty}L^2_x(I\times \R^d )}||u||^{\frac{1}{2}}_{L_t^{\infty}\dot{H}_x^{\frac{1}{2}}(I\times \R^d)}.
\end{equation}
\end{theorem}
\emph{Remark.} Generally, Theorem \ref{inter} holds for all $d\geq 1$. See \cite{Iteam1} for $d=3$, \cite{TVZ2} for $d \geq 4$ and \cite{CGT,PV} for $d=1,2$.

\begin{theorem}[Soliton-like scenario]\label{theo1}
 There are no almost periodic solutions to \eqref{01} in the setting of \eqref{so}.
\end{theorem}
\begin{proof}
We argue by contradiction argument. Assuming  $u$ is an almost periodic solution to \eqref{01} in the setting of \eqref{so}, using Theorem \ref{theo2} together with \eqref{eq:8}, we obtain
\begin{equation} \label{eq:14}
|||\nabla|^{\frac{3-d}{4}}u||_{L^4_{t,x}(\mathbb{R}\times \mathbb{R}^d)} < +\infty.
\end{equation}
\noindent Interpolating  \eqref{eq:14} with \eqref{eq:8},
\noindent We obtain
\begin{equation}\label{eq:6}
||u||_{L^{\frac{4s_c+d-3}{s_c}}_tL^{\frac{8s_c+2d-6}{2s_c+d-3}}_x(\R\times \R^d )} \lesssim 1.
\end{equation}
\noindent Moreover, using $N(t) \sim 1$, by \eqref{eq:7} and H\"older inequality,
\begin{equation}
||u(t)||_{L^{\frac{8s_c+2d-6}{2s_c+d-3}}_x(\mathbb{R}^d)} \gtrsim 1.
\end{equation}
\noindent this is a contradiction with \eqref{eq:6}.
\end{proof}

\noindent \textbf{Acknowledgments.} We highly appreciate Professor Benjamin Dodson and Professor Changxing Miao for their kind suggestions, help and discussions. Zehua Zhao is also grateful to Professor Jason Murphy and Xueying Yu for useful discussions.

 \bibliographystyle{amsplain}
 
\hfill \linebreak

\noindent \author{Chuanwei Gao}

\noindent \address{The Graduate School of China Academy of Engineering Physics, P. O. Box 2101,
Beijing, China, 100088,}

\noindent \email{canvee@163.com}\vspace{5mm}

\noindent \author{Zehua Zhao}

\noindent \address{Johns Hopkins University, Department of Mathematics, 3400 N. Charles Street, Baltimore, MD 21218, U.S.}

\noindent \email{zzhao25@jhu.edu}\\

\end{document}